\documentclass{amsart}
\usepackage{amsmath}
\usepackage{amssymb}
\usepackage{mathtools}
\usepackage{pifont}
\usepackage{multicol}
\usepackage{mathrsfs}
\usepackage{multicol}
\usepackage{amsthm}
\usepackage{float}
\usepackage{verbatim} 
\usepackage{tikz}
\usepackage{tikz-cd}
\usepackage{xcolor}
\usetikzlibrary{arrows}
\usetikzlibrary{shapes.arrows}   
\usetikzlibrary{positioning}
\usetikzlibrary {arrows.meta}
\usepackage{mathrsfs}

\usepackage{listings}
\usepackage{wasysym}
\usepackage{mathtools}
\usepackage{python}
\usepackage{enumerate}
\usepackage[numbers,comma,sort&compress]{natbib} 
\usepackage[colorlinks=true,citecolor=blue]{hyperref}

\DeclarePairedDelimiter{\ceil}{\lceil}{\rceil}
\definecolor{1}{rgb}{1,0.2,0.3}
\definecolor{2}{rgb}{0.1,0.3,0.5}
\definecolor{3}{rgb}{1,1,0}
\definecolor{4}{rgb}{255,255,255}

\usepackage{lipsum}

\newtheorem{theorem}{Theorem}[section]

\newtheorem{lemma}[theorem]{Lemma}

\theoremstyle{definition}

\theoremstyle{remark}

\begin{document}

\tikzset
{
  x=1in,
  y=1in,
}

\title[No perfect Mondrian partition for squares with side lengths $<$ 1001]{There is no perfect Mondrian partition for squares of side lengths less than 1001.}

\author{Natalia Garc\'ia-Col\'in}
\address{Natalia Garc\'ia-Col\'in, Département d'Informatique, Université Libre de Bruxelles and Department of Statistical Learning, ScaDS.AI Leipzig}
\
\email{natalia.garcia.colin@ulb.be}

\author{Dimitri Leemans}
\address{Dimitri Leemans, Département de Mathématique, Université libre de Bruxelles, and Department of Statistical Learning, ScaDS.AI Leipzig}
\email{leemans.dimitri@ulb.be}

\author{Mia Müßig}
\address{Mia Müßig, Ludwig Maximilian University of Munich, and Department of Statistical Learning, ScaDS.AI Leipzig}
\email{nienna@miamuessig.de}

\author{\'Erika Rold\'an}
\address{\'Erika Rold\'an, Max Planck Institute for Mathematics in the Sciences and 
Department of Statistical Learning, ScaDS.AI Leipzig}
\email{roldan@mis.mpg.de}



\date{\today}
\maketitle

\begin{abstract}
In mathematics, a dissection of a square (or rectangle) into non-congruent rectangles is a Mondrian partition. If all the rectangles have the same area, it is called a perfect Mondrian partition. In this paper, we present a computational result by which we can affirm that there is no perfect Mondrian partition of a length $n$ square for $n\leq 1000$. Using the same algorithm we have been able to establish that there is no perfect Mondrian partition of a $n \times m$ rectangle for $n,m \leq 400$.
\end{abstract}

\section{Introduction}

In the Journal of the Archimedeans (Cambridge University Mathematical Society), number 34, October 1971~\cite{BlancheEureka}, Blanche Decartes\footnote{Collaborative pseudonym used by the English mathematicians R. Leonard Brooks, Arthur Harold Stone, Cedric Smith, and W. T. Tutte.} published a note by the title of \emph{``Division of a square into rectangles''}  as a curious extension of their classic paper \emph{``The dissection of rectangles into squares"}  \cite{DissectionSquares}. In the former, they pose and solve the following problem:

``... Instead of dividing a square into rectangles of different sizes but all of the same shape (namely squares), \emph{one can divide it into rectangles of different shapes but of the same area}. This leads to a system of nonlinear equations. The simplest solution is that shown in Fig ...".

Their solution \cite{BlancheEureka} shows that the smallest configuration has seven rectangles of irrational side lengths. Furthermore they provide a general construction for dividing any square into $7+n$ equal area rectangles of different shapes for any $n>0$. This construction became known in the literature as a Blanche's Dissection~\cite{BlancheDissectionWolfram}.

When one adds the extra requirement that the lengths of the sides of the rectangles are integers, this dissection problem is known as the \emph{Mondrian art problem } in reference to the famous works of the dutch artist Piet Mondrian. A dissection of a square (or rectangle) into non-congruent rectangles of equal area and integer lengths is referred to as a \emph{perfect Mondrian partition}.  It is not known whether a perfect Mondrian partition exists at all. However there has been research related to the so called \emph{defect} of a Mondrian partition which is defined as the difference between the area of the smallest rectangle and the biggest rectangle. Thus, a perfect Mondrian partition has defect zero.

Traditionally the quest for perfect Mondrian partitions has focused on squares, in this context $M(n)$ denotes the \emph{Mondrian number} of a square of size $n$, defined as the minimum defect among all possible Mondrian partitions of a square of size $n$. As an $n \times n$ square can be partitioned into two non-congruent parts of roughly equal size, a natural upper bound for $M(n)$ is $n$ when $n$ is odd and $2n$ when $n$ is even.

It has been proven computationally that $M(n) > 0$ for $n\leq 65$, furthermore all $M(n)$ numbers for $n\leq 65$ are known and can be found as sequence A276523 of  \emph{The On-Line Encyclopedia of Integer Sequences}.  The researchers who found the sequence have conjectured that $M(n) \leq \lceil(n/\log(n))\rceil+3$. This bound is computationally proven to hold, through the construction of examples, up to n=96 (see~\cite{oeisMondrianDefect} and~\cite{pegg}).

A different approach,  introduced in \cite{o2018mondrian} is to study the properties of the number $x(n)= | \{x \in \mathbb{N}, |  x \leq n \text{ and } M(x) \neq 0\} |$. This number accounts for all the different numbers smaller than or equal to $n$ for which there is not a perfect Mondrian configuration. Thus, if there exists a perfect Mondrian configuration for a value $n^*$, necessarily $x(m)<m$ for any $m\geq n^*$. For example, the results of the previous paragraph imply $x(65)=65.$ In \cite{o2018mondrian} they prove a lower bound for $x(m)$, which was later improved in \cite{dalfo2021new}, namely $\frac{m}{log(m)} (1+\log(log(m))+\frac{(log(log(m)))^{2}}{2}) < x(m)$. This bound is far from being tight. 

There is other recent work in probability inspired by Mondrian partitions. For example, in  \cite{roy2008mondrian}, a family of geometric stochastic processes that resemble or are inspired by Mondrian-like configurations has been defined. These Mondrian processes give, in particular, ways of generating interesting random graph structures such as random forests \cite{o2022stochastic, martinez2023tessellation, mourtada2021amf}.

In this paper we present an algorithm that helps us conclude that there is no perfect Mondrian partition of a length n square for $n\leq 1000$  and that there is no perfect Mondrian partition of a $n \times m$ rectangle for $n,m \leq 400$. 

The results are presented as follows. 
In Section~\ref{sec:Preliminaries} we prove that a perfect Mondrian partition, if it exists, needs to have at least seven pieces. 
In Section~\ref{sec:Algorithm}, we describe the algorithms we used to check the results mentioned above.

\section{The minimum number of pieces of a perfect Mondrian partition.}\label{sec:Preliminaries}

Let $n$ and $m$ be two positive integers.
A {\em perfect Mondrian partition (PMP)} for the pair $(n,m)$ is a partition of a $n\times m$ rectangle in pairwise non-congruent rectangles that all have same area.
We include some easy observations about the structure of a PMP (if it exists). As all of these are necessary conditions for a PMP to exist they have been explicitly or implicitly used in  \cite{ o2018mondrian, dalfo2021new, dalfo2021decompositions,BlancheEureka}. We include them here for completeness and provide a new proof.

\begin{lemma} \label{lm:5pieces}
    A PMP must have at least 7 pieces.
\end{lemma}
\begin{proof}
    It is easy to see that any rectangular dissection of a rectangle with two or three pieces is such that at least two pieces will need to have the same width or height.
    
    A PMP with four different pieces has to be arranged in such a way that each piece touches exactly one corner  of the rectangle, otherwise it would contain a smaller three piece PMP.  The PMP cannot have exactly 4 pieces, as this would force them all to be congruent.

Hence if there was a PMP with five pieces , we must have a different rectangle touching every corner of the dissected rectangle and an internal rectangle that touches each piece in a configuration similar to that of Figure \ref{fig:5pMp}. As seen in the figure, for this configuration to be a PMP the heights and widths of the rectangles must satisfy the system of equations:

\begin{itemize}
     \item $a_i \times b_i = \frac{n \times m }{5}$,
     \item $a_1+a_2=n,\; a_1+a_3+a_4=n,\; a_4+a_5=n$,
     \item $b_1+b_5=m,\; b_2+b_3+b_5=m,\;b_2+b_4=m$.
\end{itemize}

Where $n,m$ are the lengths of the sides of the dissected rectangle. This system can be solved by elementary methods to show that any solution (be it in the real or integer numbers) forces the rectangles touching opposite corners to be congruent. Thus there is no PMP with exactly 5 pieces.

\begin{figure}
\begin{center}
\begin{tikzpicture}[scale=1]
    \draw (0,0) rectangle (180pt,180pt);    
    \draw (0pt,0) rectangle ++(120pt,60pt);
    \put(45,25) {$a_5 \times b_5$};
    \draw (120pt,0) rectangle ++(60pt,120pt);
    \put(135,55) {$a_4 \times b_4$};    
    \draw (180pt,180pt) rectangle ++(-120pt,-60pt);
    \put(105,145) {$a_2 \times b_2$};    
    \draw (0,180pt) rectangle ++(60pt,-120pt);
    \put(15,115) {$a_1 \times b_1$}
    \put(75,85) {$a_3 \times b_3$}
    
\end{tikzpicture}
\end{center}
\caption{The configuration of a theoretical 5-square perfect Mondrian partition.}\label{fig:5pMp}
\end{figure}
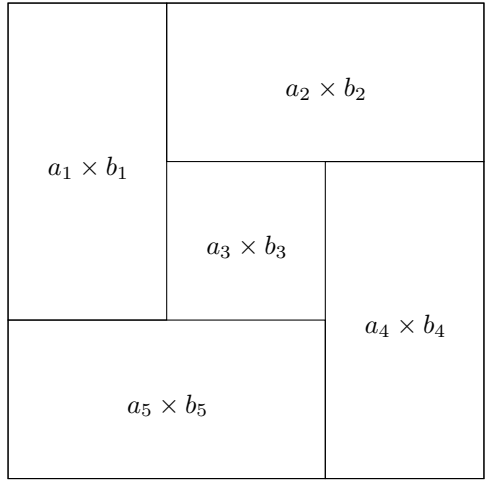

Suppose that there is a PMP with six pieces.  Then we can argue that there must be a side of the boundary  rectangle that touches at least three pieces. Firstly,  having only one piece touching one the sides of the boundary rectangle is not possible as the remaining five pieces would have to form a five piece PMP. Secondly, if exactly two pieces touched each of the sides of the boundary, we would be forced to have a configuration similar to that of Figure \ref{fig:5pMp} but with two pieces in the center; these pieces would have to be congruent. Thus there is a side of the dissected rectangle that touches at least three pieces of the partition, without loss of generality we may assume this side is the base. The rectangles touching the base all have different height and neither of them reaches the top, thus, in order to complete the partition, the one with the lowest height has to have at least two different rectangles on top and the other two have to have at least one additional rectangle on top, that makes a minimum of 7 rectangles.
\end{proof}

\section{Algorithms}\label{sec:Algorithm}

Let $n$ be the height and $m$ be the width of the rectangle we want to fill-in with non-congruent rectangles having the same area. Let $r$ be the number of pieces we want to use to cover the rectangle. Obviously, $r$ must be a divisor of $nm$ and each of the pieces will have area $\alpha := nm/r$. Additionally, as all rectangles must be pairwise non-congruent,  $\alpha$ has to have a large enough number of divisors, $\{d_1, d_2, \ldots, d_k\}$,   such that $d_i \leq n$ and $\alpha/d_i \leq m$ for all $i$ . For convenience we may assume that $d_1 < d_2 < \ldots < d_k$. 

Now consider the following set of tuples of widths and heights;
$${\mathcal P}(n,m,r) := \{d_1\times \alpha/d_1, \ldots, d_k\times \alpha/d_k\}.$$
In the vector $d_i = \alpha/d_{k+1-i}$. That is, piece $i$ is a 90 degrees rotation of piece $k+1-i,$ thus they are congruent and they cannot be simultaneously used in a successful perfect Mondrian partition. 

\subsection{(Side) Find pieces that could fill one side or the rectangle. }
The idea of this algorithm is to check, given a tuple $(n,m,r)$, if we can find a subset $\mathcal S:=\{d_{j_1}\times \alpha/d_{j_1}\ldots, d_{j_l}\times \alpha/d_{j_l}\} \subseteq \mathcal P(n,m,r)$ such that either the widths of the pieces sum to $n$, $\sum_{i:=1}^l d_{j_i} = n$, or the heights of the pieces sum to $m$, $\sum_{i:=1}^l \alpha/d_{j_1} = m$. 

Thus, this algorithm needs to solve a variation of the famously NP-complete \emph{Subset sum problem}. For small $n$ and $m$, we can use a simple backtracking approach to check if a solution exists and even find all solutions, this will become important in the next section.

We use  this approach in conjunction with Lemma~\ref{lm:5pieces} to show that there is no PMP of a $n \times n$-square for $n < 84$ . The first instance where there exists a set of at least two pieces whose width (or height) add up to $n$ occurs for $n=84$, $r = 7$, with pieces of area $\alpha = 84^2/7=1008$. In this case the set of possible tuples of width and height of  pieces is
$${\mathcal P}(84,84,7)= \{ 12 \times 84, 14 \times 72, 16 \times 63, 18 \times 48, 21 \times 42, 24 \times 36, 28 \times 28, \linebreak 36 \times 24, 42 \times 21,$$ $$\ldots , 84 \times 12 \}.$$ Figure~\ref{fillingoneside} shows the width of an $84 \times 84$ square filled bye a $12 \times 84$ rectangle and a $72 \times 14$ rectangle. A quick look at the remaining pieces allows us to see that the remaining two sides of the $84 \times 84$ square cannot be covered. This finding motivates the next algorithm.

\begin{figure}
\begin{center}
\begin{picture}(84,84)
    \put(0,0){\line(1,0){84}}
    \put(0,0){\line(0,1){84}}
    \put(84,84){\line(-1,0){84}}
    \put(84,84){\line(0,-1){84}}
    \put(12,0){\line(0,1){84}}
    \put(12,14){\line(1,0){72}}
\end{picture}
\end{center}
\caption{A $84 \times 84$ square with bottom horizontal side filled.}\label{fillingoneside}
\end{figure}
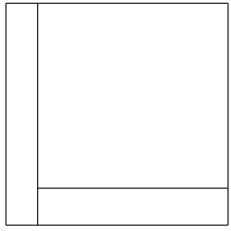

\subsection{(Perimeter) Find rectangles that can be placed around the perimeter. }
This algorithm is an extension of the previous one, its purpose is to check given a tuple $(n,m,r)$, if we can find a subset $\mathcal S:=\{d_{j_1}\times \alpha/d_{j_1}\ldots, d_{j_l}\times \alpha/d_{j_l}\} \subseteq \mathcal P(n,m,r)$ such that a combination of the widths and the heights of the pieces sum to $2(n + m)$, the perimeter of  an $n \times m$ rectangle. Unlike the previous case, here the geometry plays a bigger role as both the width and  height of the four corner pieces count for the sum, which is not the case for the rest of the pieces. 

Our algorithm checks first for which tuples $(n,m,r)$ there exists a set of pieces that can cover the perimeter of the square and later it checks if  there is still enough space remaining in the enclosed area to place at least one of the remaining pieces.

The algorithm begins  by using the ''filling one side'' algorithm described in the previous section to find all subsets of the rectangles with a total height of $n$ and all subsets of the rectangles with a total width of $m$. We call the former type a ''vertical subset'' and the latter a ''horizontal subset''. We say that a vertical subset is neighboured to a horizontal subset, if both subsets share exactly one rectangle and we save this relationship in a simple lookup table. For the purpose of filling the perimeter of the rectangle, we can use neighboured subsets to fill adjacent sides where the shared rectangle is placed in the corner.

In practice, the algorithm goes through each vertical subset and checks for \emph{every pair} of its neighboured horizontal subsets if they share another vertical subset as neighbour. In other words, the algorithm tries to find a four cycle of adjacencies that alternates between vertical and horizontal subsets. Such a set is a potential candidate for a perimeter of the $n \times m$-rectangle.  In addition, the algorithm must check that the four selected subsets do not contain duplicate (congruent) rectangles with the exception of the four corner rectangles; and that the union of the subsets contains at most $r$ rectangles (since the PMP we are seeking must have  $r$ rectangles in total). If all those conditions are met the four subsets of rectangles are a viable candidate for a filling the perimeter of a PMP of a rectangle. Note that this approach does not fix the order of the rectangles not placed in a corner.

When running the computations in the case of squares for this second part we find that the first possible $n$ (regardless of $r$) for which we can find a perimeter is 360. Therefore, we are already able to affirm that there is no PMP-square for $n<360$.

A perimeter candidate  for a square with sides $n,m=360$ has  $r = 12$,
horizontal sides covered by the rectangles 
$$[240 \times 45, 120 \times 90], [60 \times 180, 100 \times 108, 200 \times 54],$$
and vertical sides covered by the rectangles
$$[240 \times 45, 80 \times 135, 60 \times 180], [120 \times 90, 50 \times 216, 200 \times
54].$$
The remaining rectangles not used in the perimeter are:
$$36 \times 300, 40 \times 270,  48 \times 225, 72 \times 150, 75 \times 144.$$ Notice that in this instance the five remaining rectangles must be used in the partition as $r=12.$
Figure~\ref{perimeter2} shows a possible arrangement of the perimeter listed above.


\subsection{(Gap) Fitting pieces in the smallest gaps formed by the perimeter.}

Once a potential perimeter has been computed, we can look at the rectangle with the lowest height on the bottom side of the perimeter. This rectangle is sandwiched between two other rectangles, if it is not a corner rectangle, or it has a rectangle lying on top of it on one side and a rectangle with larger height on the other side. Either way there is a \emph{horizontal gap} that forms on top of said rectangle with the lowest height. This gap has a certain width that has to be filled exactly by one or several rectangles.  For example, Figure \ref{perimeter2}, shows such a gap. In that instance the width of the gap is 160 and it has to be filled in with the following remaining rectangles:
$$36 \times 300, 40 \times 270,  48 \times 225, 72 \times 150, 75 \times 144.$$
We implemented a check to see if the remaining pieces can fill this gap. Notice that this process is indifferent to which side we have chosen to be the \emph{bottom side} , thus the check can be applied to the four sides.

In the example above, the gap left to be filled on the right side is of lenght $216$, it is not hard to check that it is impossible to add $216$ exactly with a combination of the heights or the widths of the rectangles in the list above. So this case can be discarded. 

Utilising the previous three checks (Side, Perimeter, Gap) we are already able to say that there is no PMP-square for $n< 420$.

In the example above, notice that, for the gap of length $160$ left on the bottom side we do have a possible combination of heights and widths of rectangles adding to $160$, namely 40+48+72. However, as the height of the rectangle with width of $40$ is $270$ a quick look at Figure \ref{perimeter2} allows us to see that this rectangle would intersect the rectangles in the top side of the square. This phenomenon has inspired the next check.

\begin{figure}
\begin{center}
\begin{tikzpicture}
    \draw (0,0) rectangle (360pt,360pt);
    
    \draw (0pt,0) rectangle ++(240pt,45pt);
    \put(100,20) {$240 \times 45$};
    \draw (240pt,0) rectangle ++(120pt,90pt);
    \put(280,40) {$120 \times 90$};
    
    \draw (0pt,180pt) rectangle ++(60pt,180pt);
    \put(10,270) {$60 \times 180$};

    \draw (60pt,360pt) rectangle ++(100pt,-108pt);
    \put(85,300) {$100 \times 108$};
    \draw (160pt,360pt) rectangle ++(200pt,-54pt);
    \put(235,330) {$200 \times 54$};
    
    \draw (0,45pt) rectangle ++(80pt,135pt);
    \put(25,110) {$80 \times 135$};
    
    \draw (360pt,90pt) rectangle ++(-50pt,216pt);
    \put(315,200) {$50 \times 216$};

    \draw[<->] (80pt,70pt) -- (240pt,70pt);
    \put(150,75) {$160$};

    \draw[<->] (290pt,90pt) -- (290pt,306pt);
    \put(270,190) {$216$};
    
\end{tikzpicture}
\caption{An example of a perimeter with $n=360$ with a hole to fit.}\label{perimeter2}
\end{center}
\end{figure}
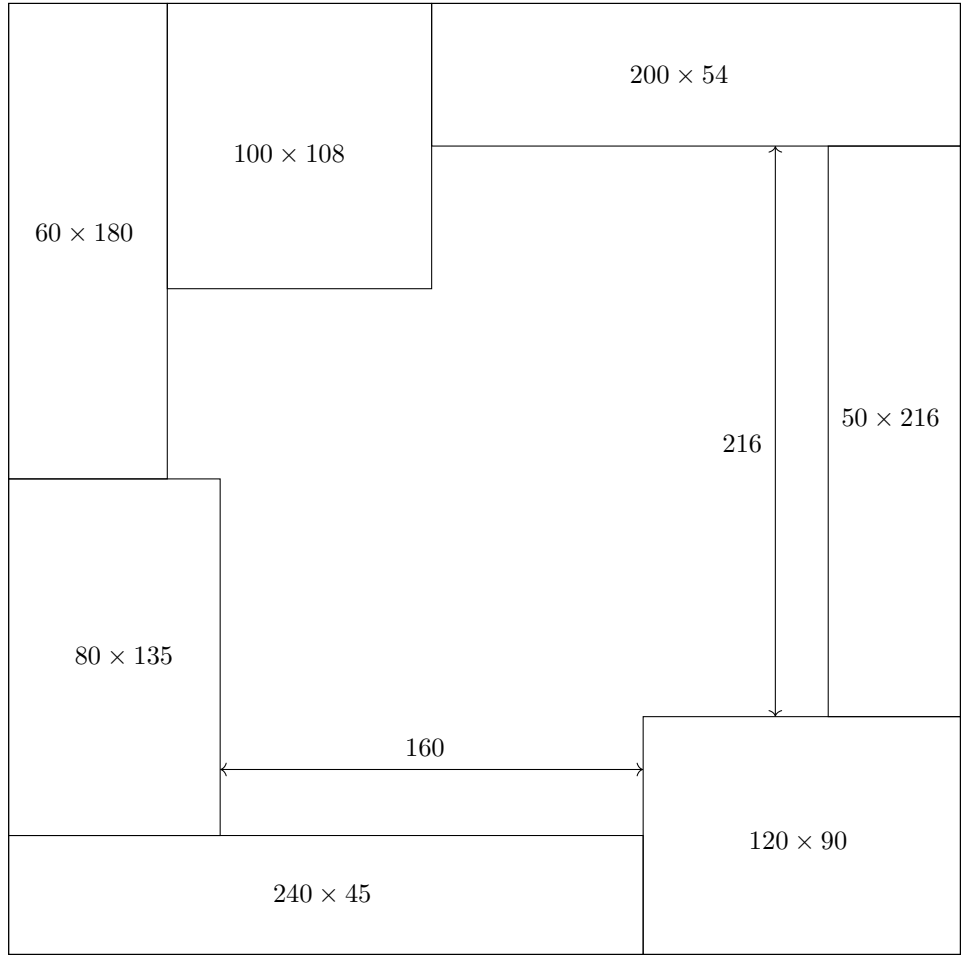

\subsection{(Hole) Are the remaining rectangles small enough to fit in the hole formed by the perimeter? }

Once a potential perimeter has been computed, we can easily obtain an upper bound on the size of the hole formed by the perimeter. An upper bound on the width of the hole, the \emph{mwh}, can be obtained by deducting from $n$ the two minimum widths of the rectangles in each of the two vertical subsets. Similarly, an upper bound on the height of the hole, the \emph{mhh}, can be computed by deducting from $m$ the two minimum heights of the rectangles in each of the two horizontal subsets.

We  know that every rectangle not in the perimeter has to have a   width smaller or equal than the \emph{mwh} and a height smaller or equal than the \emph{mhh}. Thus, if not enough of the remaining rectangles satisfy this property, we can eliminate the case. For instance, in the example that we have been using $mwh=250$ and $mhh=261$ (see Figure~\ref{perimeter3}), but we also know that we need to place all 5 rectangles not in the perimeter in the hole. However, in the list of remaining rectangles, we have are rectangle of size $36 \times 300$ which clearly cannot fit horizontally or vertically in the hole. Hence, we discard this case.

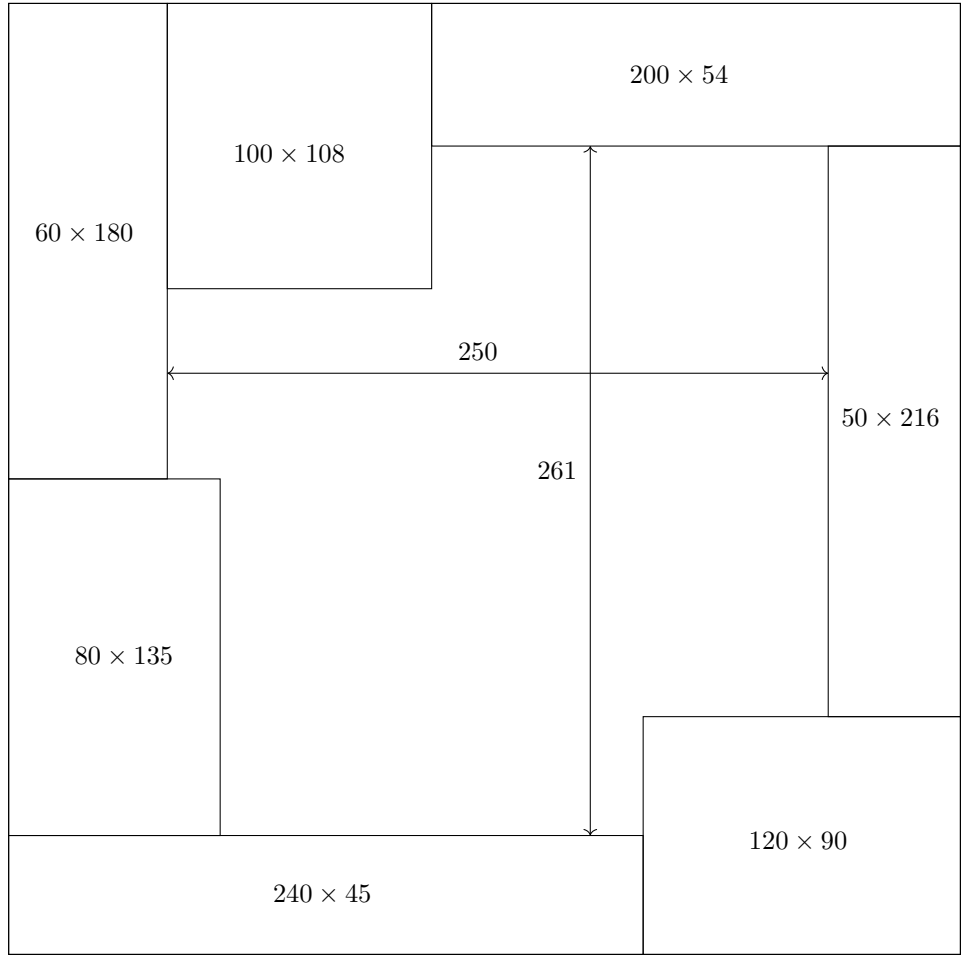
\begin{figure}
\begin{center}
\begin{tikzpicture}
    \draw (0,0) rectangle (360pt,360pt);
    
    \draw (0pt,0) rectangle ++(240pt,45pt);
    \put(100,20) {$240 \times 45$};
    \draw (240pt,0) rectangle ++(120pt,90pt);
    \put(280,40) {$120 \times 90$};
    
    \draw (0pt,180pt) rectangle ++(60pt,180pt);
    \put(10,270) {$60 \times 180$};

    \draw (60pt,360pt) rectangle ++(100pt,-108pt);
    \put(85,300) {$100 \times 108$};
    \draw (160pt,360pt) rectangle ++(200pt,-54pt);
    \put(235,330) {$200 \times 54$};
    
    \draw (0,45pt) rectangle ++(80pt,135pt);
    \put(25,110) {$80 \times 135$};
    
    \draw (360pt,90pt) rectangle ++(-50pt,216pt);
    \put(315,200) {$50 \times 216$};

    \draw[<->] (60pt,220pt) -- (310pt,220pt);
    \put(170,225) {$250$};
    \draw[<->] (220pt,45pt) -- (220pt,306pt);
    \put(200,180) {$261$};
    
\end{tikzpicture}
\caption{An example of a perimeter with $n=360$.}\label{perimeter3}
\end{center}
\end{figure}

Saddly, adding this check does not allow us to immediately improve the previous bound for $n$ in the case of squares. 

However, implementing all these checks together, we can greatly reduce the number of cases where a  \emph{brute force computation} of all possible cases is needed. For the case of squares, only the following values of $n<1001$ could not be ruled out by the Side, Perimeter, Gap and Hole checks:
\begin{equation} \label{list:remainingValues}
    \{420, 480, 630, 660, 720, 780, 840, 900, 924, 960, 990\}
\end{equation}
 
For these cases we implemented a backtracking algorithm that we describe in the next section.

\subsection{The backtracking algorithm}
This simple backtracking algorithm will take as input the width and height of the base rectangle $n$ and $m,$ respectively, and output a PMP if there is one and FALSE otherwise.

What follows is an outline of the logic behind the algorithm. The full code is accessible in the following repository.
\begin{center}
\url{https://github.com/PhoenixSmaug/Mondrian}
\end{center}
Given a pair $(n,m)$ of integers, start by generating the set of all divisors of $n \times m$,  and for each divisor $r\geq 7$ generate the set of all possible tuples of widths and heights $${\mathcal P}(n,m,r) := \{d_1\times \alpha/d_1, \ldots, d_k\times \alpha/d_k\}$$ with product equal to $\frac{n \times m}{r},$,   such that $d_i \leq n$ and $\alpha/d_i \leq m$ for all $i.$ If the cardinality of ${\mathcal P}(n,m,r) $ is less than $r$ the case can be discarded. Else, proceed to use the rectangles to fill in the base of the $n \times m$--rectangle from left to right. 

The data structure we used to keep track of the locations, inside the main rectangle, that are covered by rectangles at any given time during the execution is an integer vector, $V,$ of size $n$,  that contains, at position $i$ the height already covered at that width, following ideas of~\cite{Chazelle}.

At any time during the execution, the position chosen as the bottom left corner of the next rectangle to be placed is  the smallest index $i \in [n]:=\{1, \ldots,n\}$ such that $V[i]= \min_{i \in [n]} \{V[i]\}.$ It follows that, when a piece of width $w$ and height $h$ has to be placed in the main rectangle, we first need to check that it can indeed be placed in the horizontal gap available, i.e. $V[j]=V[i]$ for all $j \in \{i, \ldots, i+w-1\}$. If that check succeeds, then we fill the vector $V$ with value $h+V[j]$ for entries $j=i$ to $i+w-1.$ We must also keep track of the rectangles that have already been used in order to make sure we do not reuse them or their 90-degree-rotations, and also to be able to backtrack easily.

In the case where $n=m$, we are dealing with a square and we can take advantage of the fact that it has a group of symmetries of order 8 so, we can potentially avoid making the same computation 8 times. In order to speed up the search,  we make use of two of the symmetries of the square, the \emph{vertical mirror symmetry} defined by the vertical line through the midpoints of the two horizontal sides of the square, and a \emph{diagonal symmetry}, defined by the line through the bottom left corner and the top right corner of the square, see Figure~\ref{fig:symmetries}). This will allow us to avoid making the same computation 4 times. We believe it will rarely happen that the top left corner will be reached by the backtracking algorithm and hence, implementing a check to use a third symmetry to reduce even more the number of cases to be considered will end up increasing computing time.
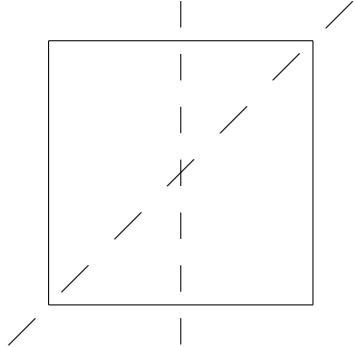
\begin{figure} \label{fig:symmetries}
\begin{picture}(140,140)
\put(15,15){\line(1,0){100}}
\put(15,15){\line(0,1){100}}
\put(115,115){\line(-1,0){100}}
\put(115,115){\line(0,-1){100}}
\multiput(0,0)(20,20){7}{\line(1,1){10}}
\multiput(65,0)(0,20){7}{\line(0,1){10}}
\end{picture}
\caption{The two symmetries of the square used to reduce the search tree.}
\end{figure}
The chosen diagonal symmetry of the square allows to restrict the search to the cases where the width and the height of the first rectangle placed satisfies $d_i \leq \alpha/d_i.$ This holds because if there was a PMP corresponding to starting with such rectangle there is a PMP corresponding to starting with the rotation of  such rectangle. Therefore, we can restrict the first choice to the set of rectangles $\{d_1\times \alpha/d_1, \ldots, d_{\ceil{\frac{n}{2}}}\times \alpha/d_{\ceil{\frac{n}{2}}}\}.$
The vertical mirror symmetry implies that,  when considering a piece in the bottom left corner that has width $d_i$ and height $\alpha/d_i$, then the only pieces to be considered for placement at the bottom right corner must have width $d_j$ and height $\alpha/d_j$ with $i < j$. Furthermore, as the rectangles with index $j \geq k+1-i$ are rotations of the rectangles with index  $j \leq i,$ the diagonal symmetry allows to conclude that we can further restrict the rectangle to be placed in the bottom right corner to have an index satisfying $j < k+1-i$.

When $n\neq m$, the diagonal symmetry cannot be used anymore but the vertical one can. However, in our computations we did not use the vertical symmetry as it was not drastically reducing the computing time.

\subsection{Benchmarks}
The backtracking algorithm described is very slow, as such, it is better to use it only for the values of $n<1001$ for which we did not obtain a negative answer with the Side, Perimeter, Gap and Hole algorithms. That is, the values in Equation \ref{list:remainingValues}. Out of the values in the list, the hardest cases to compute for the backtracking algorithm were for $n=840$ and $r=20$ or $21$. 

\begin{figure}[h]
    \centering
    \includegraphics[width=1\textwidth]{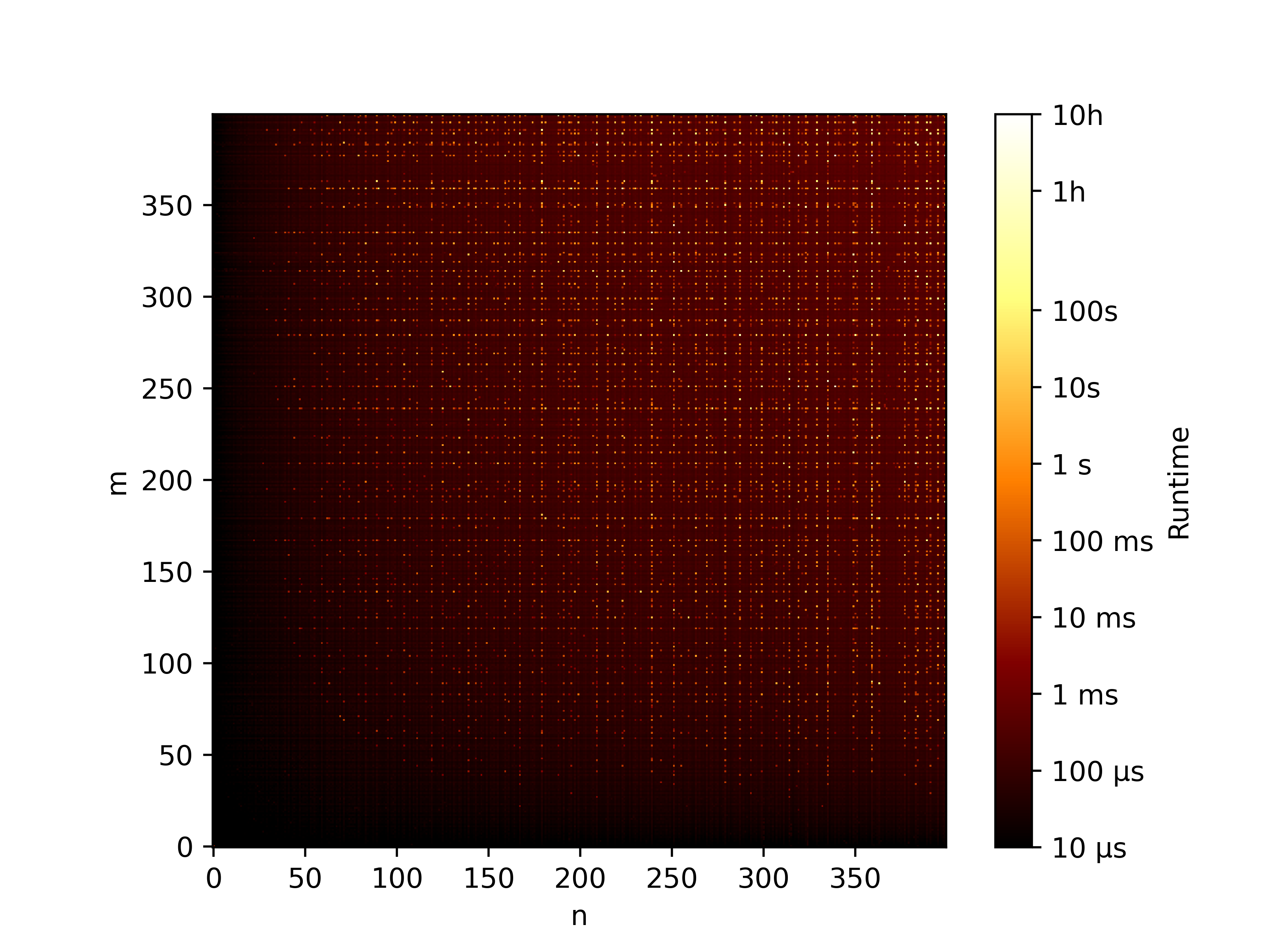}
    \caption{Average run-time after 10 samples of checking for each $n \times m$ rectangle with $n, m \leq 400$ if a perfect Mondrian partition exists. Since you can always rotate the rectangle such that $n \geq m$, the data is mirrored along the diagonal. The data was collected using the Julia library ''BenchmarkTools.jl'' on a PC with an ''AMD Ryzen 9 5950X''.}
\end{figure}

The logarithmic colour plot visualizes the exponential run-time of the problem, which depends strongly on the number of divisors of $nm$. This behavior produces distinct lines of higher run-time for side lengths with many divisors, such as $240$, $256$, and $360$. On the lower left one can also see a quarter circle where the product $nm$ is just too small. It should be noted that the number of divisors is only an orientation of the run-time and, as with many other NP-complete problems, it is very difficult to predict how difficult a particular case is.

The cutoff $n, m \leq 400$ was chosen because the run-time increases even further and the one-time check for each $420 \times m$-rectangle with $m \leq 420$ already takes more than two weeks. The $840 \times 840$ rectangle has a run-time of about 400 days, although it should be noted that this is only an estimate of the program itself and we checked for a Mondrian partition in this case using parallel computation. To do this, we went through all the combinations for the placement of the first two rectangles and then dispatched a thread to check if the rest of the rectangles fit.
\section*{Acknowledgements}
Dimitri Leemans acknowledges financial support from the Communauté Française Wallonie Bruxelles through an Action de Recherche Concertée grant.

\bibliographystyle{abbrvnat}

\end{document}